\documentclass[11pt,a4paper]{amsart}
\usepackage[T1]{fontenc}
\usepackage{lmodern,amsmath,amssymb}
\usepackage{xcolor}
\usepackage[margin=30mm]{geometry}
\usepackage[hypertexnames=false,colorlinks=true,linkcolor=blue,citecolor=blue,urlcolor=blue]{hyperref}
\hypersetup{
 pdftitle={Disproving Conjectures on the Davenport and Narkiewicz Constants},
 pdfauthor={Guoqing Wang},
 pdfsubject={Davenport and Narkiewicz Constants}}
\newcommand{\dd}{\mathsf d}
\newcommand{\DD}{\mathsf D}
\newcommand{\kk}{\mathsf k}
\newcommand{\rr}{\mathsf r}
\newcommand{\NN}{\mathbb N}
\newcommand{\Nark}{\mathsf N}
\newcommand{\ZZ}{\mathbb Z}
\newcommand{\FF}{\mathbb F}
\newcommand{\cF}{\mathcal F}
\DeclareMathOperator{\ord}{ord}
\DeclareMathOperator{\supp}{supp}
\DeclareMathOperator{\vterm}{v}
\DeclareMathOperator{\lcm}{lcm}

\newtheorem{theorem}{Theorem}[section]
\newtheorem{conjecture}{Conjecture}
\newtheorem{proposition}[theorem]{Proposition}
\newtheorem{lemma}[theorem]{Lemma}
\newtheorem{corollary}[theorem]{Corollary}
\theoremstyle{definition}
\newtheorem{example}[theorem]{Example}
\newtheorem{problem}[theorem]{Problem}
\theoremstyle{remark}
\newtheorem{remark}[theorem]{Remark}
\makeatletter
\renewcommand{\@setkeywords}{%
  \textit{Keywords:}\enspace\@keywords\@addpunct.}
\makeatother

\makeatletter
\@ifundefined{subjclassname@2020}{%
 \@namedef{subjclassname@2020}{\textup{2020} Mathematics Subject Classification}%
}{}
\makeatother

\title[Disproving Conjectures on the Davenport and Narkiewicz Constants]{Disproving Conjectures on the Davenport and Narkiewicz Constants}

\thanks{Remark: The uniform upper bound
$D(G)\le \frac{16}{5}r(G)\exp(G)$ was already stated
and proved in arXiv:2609.29878v2
(29 September 2026, Theorem~6.2) and remains unchanged.
The unboundedness result is strengthened: for every
fixed $r\ge 5$, the difference $D(G)-D^*(G)$ is unbounded
among finite abelian groups of rank $r$, improving
the condition $r\ge 8$ in versions~2.
To keep the paper concise and focused on its main themes,
Section~7 of versions~3 and~4, which proves the
subgroup-divisibility characterization conjecture has been omitted.
The title has been changed from
\emph{Disproof of a Conjectured Upper Bound for the
Davenport Constant} to
\emph{Disproving Conjectures on the Davenport and
Narkiewicz Constants}.}
\author{Guoqing Wang}
\address{School of Mathematical Sciences\\
Tiangong University\\
Tianjin 300387, PR China}
\email{gqwang1979@aliyun.com}

\begin{document}
\begin{abstract}
The Davenport constant $\DD(G)$ of a finite abelian
group $G$ is a classical invariant arising from the
arithmetic of ideal class groups of number fields.
Closely related to the Davenport constant is the first
Narkiewicz constant $\Nark_1(G)$, which arises in the
quantitative study of algebraic integers and their
factorizations. Gao and Geroldinger \cite[Conjecture~3.7]{GG06} conjectured in 2006
that
$\DD(G)\leq \sum_{i=1}^r n_i,$
while Narkiewicz and \'Sliwa \cite{NS82} conjectured in
1982 that
$\Nark_1(G)=\sum_{i=1}^r n_i,$
where $n_1,\ldots,n_r$ are the invariant factors of $G$.
Both are central conjectures concerning arithmetic invariants
arising from the theory of number fields. In this paper,
we disprove both conjectures. Moreover, we
show that, for every fixed $r\geq5$, both
$\DD(G)-\sum_{i=1}^r n_i$ and
$\Nark_1(G)-\sum_{i=1}^r n_i$
are unbounded over finite abelian groups of rank $r$.

We develop a probabilistic representation technique to establish
a uniform upper bound for the Davenport constant:
$$
\DD(G)\leq\frac{16}{5}\,\rr(G)\exp(G).
$$
Our bound gives the first improvement of the classical
van Emde Boas--Kruyswijk estimate: $\DD(G)\leq\exp(G)
\left(1+\log\frac{|G|}{\exp(G)}\right)
\leq\exp(G)\bigl(1+(\rr(G)-1)\log\exp(G)\bigr)$, by removing the logarithmic dependence on the exponent
from this classical estimate, replacing it with the absolute
constant $16/5$. Their bound stood unimproved since its establishment in 1969.
The existence of $C=16/5$ also brings us closer to the most challenging
conjecture on the Davenport constant, namely the
Emde Boas--Kruyswijk Conjecture \cite{vEBK69} in 1969, which
predicts that
$\DD(C_n^r)=r(n-1)+1.$
As a further
application of our results, we disprove a conjecture of Girard
\cite[Conjecture~1.2]{Girard08} in 2008 concerning cross numbers,
an invariant with deep roots in
the arithmetic of number fields and Krull monoids.
\end{abstract}
\maketitle
\date{}
\section{Introduction}

The invariant now known as the Davenport constant
$\mathsf D(G)$ of a finite abelian group $G$ was already
studied by Rogers \cite{Rogers63} in 1963 and was later
popularized by H. Davenport.
It has its origins in the arithmetic of ideal class groups:
if $G$ is the ideal class group of a number field,
then $\mathsf D(G)$ is the maximal number of prime ideals,
counted with multiplicity, occurring in the prime ideal
factorization of a principal ideal generated by an
irreducible element.
It has since developed deep connections with the factorization theory of Krull monoids, algebraic and classical number theory, coding theory; see, for instance, \cite{AlGrPome, GG06,GeroMathZ, GeroGryn09, GGZ26,GHK06,GZ20, SchmidCoding}.

Determining its exact value, or even obtaining sharp general bounds, is therefore a fundamental problem with implications for several major areas of mathematics. A striking example of its number-theoretic significance is the proof by Alford, Granville, and Pomerance \cite{AlGrPome} that there exist infinitely many Carmichael numbers, in which an upper bound for the Davenport constant plays a crucial role.

Despite decades of substantial progress and the development of
powerful methods, the exact determination of $\DD(G)$ remains
notoriously difficult. Explicit formulas are known for $p$-groups
and groups of rank at most two, as well as for several very specific
families, but the general case remains largely elusive; see \cite{Gao2000, GG06,Olson69I,Olson69b,Schmid11}.
The fundamental theorem of finite abelian groups gives a
decomposition of $G$ as a direct sum of cyclic subgroups:
$$
G=C_{n_1}\oplus\cdots\oplus C_{n_r},
\qquad 1<n_1\mid\cdots\mid n_r,
$$
where $r=\rr(G)$ denotes the rank of $G$.
A standard construction
gives the classical lower bound
$$
\DD(G)\ge \DD^*(G):=
1+\sum_{i=1}^r(n_i-1).
$$
For a long time, it was a prevailing expectation that the difference $\mathsf D(G)-\mathsf D^*(G)$ should admit an upper bound depending only on the rank of $G$. This expectation was formulated explicitly in the following conjecture of Gao and Geroldinger.

\begin{conjecture}[Gao--Geroldinger, {\cite[Conjecture~3.7]{GG06}}]\label{conj:upper-bound}
Every nontrivial finite abelian group $G$ satisfies
$$\DD(G)-\DD^*(G)\le\rr(G)-1.$$
\end{conjecture}

This conjecture is closely connected with the arithmetic of algebraic number fields. In fact, it is implied by the following conjecture of Narkiewicz and \'{S}liwa, proposed in 1982, concerning the first Narkiewicz constant, which arises in the asymptotic counting of algebraic integers admitting unique factorization into irreducibles. We introduce this constant in Section~\ref{sec:narkiewicz}.

\begin{conjecture}[Narkiewicz--\'{S}liwa,
{\cite{NS82}}]\label{conj:narkiewicz-sliwa} Let $G=C_{n_1}\oplus\cdots\oplus C_{n_r}$ be a nontrivial finite abelian group, where $1<n_1\mid\cdots\mid n_r$. Then
$\Nark_1(G)=\sum_{i=1}^r n_i.$
\end{conjecture}

The inequality $\mathsf D(G)\le \Nark_1(G)$, established in Section~\ref{sec:narkiewicz}, shows that the Narkiewicz--\'{S}liwa conjecture would imply Conjecture \ref{conj:upper-bound}. Hence, a possible rank-dependent upper bound for this difference $\DD(G)-\DD^*(G)$ is
therefore of particular importance, with implications
for problems in algebraic and classical number theory.
Geroldinger, Grynkiewicz, and Zhong also highlighted this
question as a challenging open problem in their monograph
published in 2026.

\begin{problem}[Geroldinger--Grynkiewicz--Zhong,
{\cite[Appendix~B, Problem~4(a)]{GGZ26}}]\label{prob:rank-bound}
Does there exist a function $f:\NN\to[0,\infty)$ such that
$$
\DD(G)-\DD^*(G)\le f(\rr(G))
$$
for every nontrivial finite abelian group $G$? In particular, can $f$ be chosen linear?
\end{problem}

In this paper, we prove that, for every integer $r\ge 5$,
the difference $\DD(G)-\DD^*(G)$ is unbounded among finite
abelian groups $G$ of rank $r$.
Consequently, there is no finite-valued function
$f:\mathbb N\to[0,\infty)$ such that
$\DD(G)\le \DD^*(G)+f(\rr(G))$
for every nontrivial finite abelian group $G$, giving a
negative answer to Problem~\ref{prob:rank-bound}.
As a corollary, our result disproves Conjectures~1 and~2,
thereby refuting a long-standing expectation in the
study of the Davenport constant and the arithmetic of
ideal class groups. Furthermore, we apply our construction to disprove a conjecture of
Girard in 2008 concerning cross number (Conjecture \ref{conj:girard} in Section \ref{sec:cross}), an invariant
deeply rooted in
the arithmetic of number fields and Krull monoids.

A natural next question is whether one can obtain effective upper bounds for $\mathsf D(G)$. A landmark result in this direction was established by van Emde Boas and Kruyswijk \cite{vEBK69} in 1969:
$$
\DD(G)\le
\exp(G)\left(1+\log\frac{|G|}{\exp(G)}\right).
$$
Remarkably, no uniform improvement of this estimate had been obtained for general finite abelian groups since its establishment in 1969.

In this paper, we develop a new probabilistic representation method and combine it with the classical representation approach of van Emde Boas and Kruyswijk, as well as the group algebra framework developed by Gao and Geroldinger \cite{GaoGeroldinger2007}, \cite[Section 5.5]{GHK06} and \cite[Section 2.2]{GeroldingerRuzsa09} for studying additive combinatorial problems and related invariants, to establish the following universal upper bound:
$$
\mathsf{D}(G)\le \frac{16}{5}\,\mathsf{r}(G)\exp(G)
$$
for every nontrivial finite abelian group $G$.

Our result provides the first uniform improvement yielding a bound linear in the product of the rank and exponent, replacing the logarithmic dependence in the classical estimate with the absolute constant $16/5$.
Writing $G=C_{n_1}\oplus\cdots\oplus C_{n_r}$, where
$1<n_1\mid\cdots\mid n_r$, the bound of van Emde Boas and
Kruyswijk takes the form
$\mathsf D(G)\le n_r\left(1+\sum_{i=1}^{r-1}\log n_i\right).$

The bound $\frac{16}{5}\,\mathsf{r}(G)\exp(G)$ also has implications for some fundamental
open problems concerning the Davenport constant. We begin
with the following long-standing Conjecture \ref{Con:EmdeKry}, first raised
by van Emde Boas and Kruyswijk \cite{vEBK69} in 1969. In particular, replacing $16/5$ by $1$ would yield
$\mathsf D(C_n^r)\le rn,$
which is already close to the conjectured exact value
$r(n-1)+1$.
Before our work, however, no constant $C$ giving $\mathsf D(C_n^r)\le C rn$.

\begin{conjecture}[Emde Boas--Kruyswijk Conjecture, 1969]
\label{Con:EmdeKry}
For all integers $n,r\ge1$,
$$
\DD(C_n^r)=r(n-1)+1.
$$
\end{conjecture}

A further contribution of this work is the development of
a new probabilistic representation method. We use this
method to establish the above bound for the Davenport
constant.  More broadly, it provides a probabilistic
framework for studying algebraic and additive problems
over finite groups.

The paper is organized as follows. Section~\ref{sec:construction}
gives the general construction and disproves the Gao--Geroldinger conjecture. Sections~\ref{sec:narkiewicz}
and~\ref{sec:cross} disprove the Narkiewicz--\'{S}liwa and Girard
conjectures, respectively. Section~\ref{sec:rank-exponent} develops the
probabilistic representation method and
establishes a uniform upper bound for the Davenport constant.

\section{Notation and preliminaries}\label{sec:notation}

We write $\mathbb{N}=\{1,2,\ldots\}$ and $\mathbb{N}_0=\mathbb{N}\cup\{0\}$. For integers $a\le b$, we use $[a,b]$ to denote the set of integers $\{a,a+1,\ldots,b\}$. Let $G$ be a finite abelian group.
Subgroups $H_1,\ldots,H_t$ of $G$ are called \emph{independent} if the equality
$$
h_1+\cdots+h_t=0,\qquad \text{where } h_i\in H_i,
$$
implies that $h_i=0$ for every $i\in[1,t]$.
Equivalently, their sum is the internal direct sum
$H_1\oplus\cdots\oplus H_t$.
A set $\{g_1,\ldots,g_t\}$ of nonzero elements of $G$
is called \emph{independent} if the cyclic subgroups
$\langle g_1\rangle,\ldots,\langle g_t\rangle$
are independent.
A \emph{basis} of $G$ is an independent generating set.

We need notation and terminologies on sequences, and use the standard language of sequences over groups; see \cite{GG06}
 and \cite[Chapter 5]{GHK06}.  A sequence over
$G$ is an element of the free abelian monoid $\mathcal F(G)$ and will be written in the form
$$S= \prod\limits_{g\in G} g^{\mathsf v_g(S)} = g_1\boldsymbol{\cdot}\ldots\boldsymbol{\cdot}g_\ell \,,$$
where $\mathsf v_g(S)\in\mathbb N_0$ is the {\sl multiplicity} of $g$ in $S$.
Its {\sl length} and {\sl support} are
$$|S|=\sum_{g\in G}\vterm_g(S),\qquad
 \supp(S)=\{g:\vterm_g(S)>0\}.$$
We denote
$$\sigma(S)=\sum_{i=1}^{\ell}g_i$$ to be the sum of all terms from the sequence $S$.
We also denote the set of nonempty subsums of $S$ by
$$\Sigma(S)=\left\{\sum_{i\in I}g_i:
                    \emptyset\ne I\subseteq[1,\ell]\right\}.$$
Denote
$$
\Sigma_0(S)=\Sigma(S)\cup \{0\}.
$$
A sequence $T \in \mathcal F(G)$ is called a {\it subsequence } of $S$ and is denoted by $T \mid S$ if  $\mathsf v_g(T) \le \mathsf v_g(S)$ for all $g\in G$.
The empty sequence is denoted by $1$, which has length
zero and sum zero. The sequence $S$ is called \emph{zero-sum free} if $0\notin\Sigma(S)$. It is called a \emph{minimal zero-sum sequence} if $S$ is nonempty, $\sigma(S)=0$, and every nonempty subsequence $T\mid S$ with $T\ne S$ satisfies $\sigma(T)\ne0$.

The Davenport constant $\DD(G)$ is the least positive integer
$\ell$ such that every sequence of length $\ell$ over $G$
has a nonempty zero-sum subsequence. The small Davenport constant $\dd(G)$ is the maximum
length of a zero-sum free sequence over $G$. It is well known that
$$\DD(G)=\dd(G)+1,$$ see \cite[Section 2]{GG06}. For $n\in\NN$, we denote $C_n$ to be the cyclic group of order $n$.
By the fundamental theorem of finite abelian groups, every nontrivial finite abelian group $G$ admits a decomposition
$$
G=C_{n_1}\oplus\cdots\oplus C_{n_r},
\qquad 1<n_1\mid\cdots\mid n_r,
$$
where $r=\rr(G)$ is the rank of $G$. The order and exponent of $G$ are denoted by $|G|$ and $\exp(G)=n_r$, respectively.
Denote
$$\dd^*(G)=\sum_{i=1}^r(n_i-1),\qquad
\DD^*(G)=\dd^*(G)+1.$$
We call
$$
\delta(G)=\DD(G)-\DD^*(G)=\dd(G)-\dd^*(G)$$
the \emph{Davenport excess} of $G$.

For a sequence $S=g_1\boldsymbol{\cdot}\ldots\boldsymbol{\cdot}g_\ell$ over $G\setminus\{0\}$, its
\emph{cross number} is
$$
\kk(S)=\sum_{i=1}^{\ell}\frac1{\ord(g_i)}.
$$
The little cross number $\kk(G)$ is the maximum of $\kk(S)$
over all zero-sum free sequences over $G$. Thus $\dd(G)$ and
$\kk(G)$ measure ordinary and weighted length under the same
zero-sum restriction.

\section{Disproof of Conjecture~\ref{conj:upper-bound}}\label{sec:construction}

In this section, we shall disprove Conjecture~\ref{conj:upper-bound}, and furthermore, we shall give a way to construct new groups $G$ with rank fixed but Davenport excess $\delta(G)$ increasing. Hence, the section is divided into 2 subsections.

\subsection{An explicit family with unbounded excess}\label{subsec:explicit}

The classical basis sequence $\prod_{i=1}^r e_i^{n_i-1}$
is zero-sum free and gives $$\dd(G)\ge\dd^*(G).$$
Equality holds for $p$-groups and groups of rank at most two;
see \cite[Chapter~5]{GHK06} and
\cite[Theorems~3.3.6, 3.4.8, and~3.5.3; Section~3.10]{GGZ26}
for the classical results and their history.

Equality fails in general. Baayen's example $C_2^4\oplus C_6$
was found in \cite[Theorem 8.1]{vEB69}, and direct powers
give unbounded excess while increasing the rank
\cite[Corollary~8.2]{vEB69}.
Geroldinger and Schneider \cite{GS92} found rank-four examples
with positive excess. Mazur \cite{Mazur92} proved logarithmic
lower growth in the rank for $C_2^{r-1}\oplus C_{2k}$ with
fixed odd $k>1$. These results do not give unbounded Davenport excess
at fixed rank.

Conjecture~\ref{conj:upper-bound}, proposed in
\cite[Conjecture~3.7]{GG06} and subsequently also discussed in \cite[Section~2]{BB07},
would bound the excess $\delta(G)$ by $\rr(G)-1$ independently of the orders
of the cyclic factors. In this section, we shall disprove this conjecture by the following theorem.

\begin{theorem}\label{thm:main}
Let $r\ge8$ and $m\ge1$ be integers, and put
$$G_{m,r}=C_m^{\,r-5}\oplus C_{2m}^{\,4}\oplus C_{6m}.$$
Then
$\delta(G_{m,r})=\dd(G_{m,r})-\dd^*(G_{m,r})\ge m.$
\end{theorem}

To prove Theorem \ref{thm:main}, we need the following lemma.

\begin{lemma}\label{lem:binary}
Let $\{v_1,v_2,v_3,v_4\}$ be a basis of $\FF_2^4$, and put
$\Omega=\{v_1,v_2,v_3,v_4,\,
 v_2+v_3+v_4,\,
 v_1+v_3+v_4,\,
 v_1+v_2+v_4\}.$
Every nonempty subset of $\Omega$ with sum zero has
cardinality four.
\end{lemma}
\begin{proof}
For a vector expressed as a linear combination of the chosen basis vectors $v_1,v_2,v_3,v_4$, its weight is the number of nonzero coefficients in that linear combination. Thus, $\Omega$ consists of the eight vectors of odd weight, excluding $w=v_1+v_2+v_3$.
With respect to the chosen basis, we define the linear map
$\lambda:\mathbb F_2^4\longrightarrow\mathbb F_2$
by
$$\lambda(a_1v_1+a_2v_2+a_3v_3+a_4v_4)
=a_1+a_2+a_3+a_4\in\mathbb F_2.$$

Let $B\subseteq\Omega$ be a {\sl nonempty} subset such that
$\sum_{v\in B}v=0.$
Since $\lambda(v)=1$ for every $v\in\Omega$, applying $\lambda$ gives
$$
0=\lambda\left(\sum_{v\in B}v\right)
=\sum_{v\in B}\lambda(v)
=|B|\cdot 1
\quad\text{in }\mathbb F_2.
$$
Thus, $|B|$ is even, and so $$|B|\in\{2,4,6\}$$ since $|\Omega|=7$.
The case $|B|=2$ is impossible, since two vectors in
$\FF_2^4$ sum to zero only when they are equal. It is easy to verify that
$$
\sum_{v\in\Omega}v=v_1+v_2+v_3=w.
$$
If $|B|=6$, let $u$ be the unique element of
$\Omega\setminus B$. Then
$w=\sum_{v\in\Omega}v
=\sum_{v\in B}v+u=u,$
contradicting $u\in\Omega$ and $w\notin\Omega$.
Therefore $|B|=4$, completing the proof of the lemma.
\end{proof}

Now we are in a position to show Theorem \ref{thm:main}.

\begin{proof}[Proof of Theorem \ref{thm:main}]
We first construct a zero-sum free sequence of length $18m-8$ over the group
$$B_m=C_m^3\oplus C_{2m}^4\oplus C_{6m}.$$
Let $z_1,z_2,z_3,y_1,y_2,y_3,y_4,e$ generate the respective cyclic direct summands of $B_m$, with
$$
\ord(z_1)=\ord(z_2)=\ord(z_3)=m,
$$
$$
\ord(y_1)=\ord(y_2)=\ord(y_3)=\ord(y_4)=2m,
\qquad \ord(e)=6m.
$$
Define
\begin{align}
q_j&=y_j+2e &&\text{for }j\in[1,4],\label{eq:q-first}\\
q_{4+i}&=z_i+\sum_{\substack{j\in[1,4]\\j\ne i}}y_j+2e
&&\text{for }i\in[1,3].\label{eq:q-last}
\end{align}
Put
$$S_m=e^{\,4m-1}\prod_{j=1}^7q_j^{\,2m-1}.$$
Clearly,
\begin{equation}\label{eq:lengthofSm}
|S_m|=(4m-1)+7(2m-1)=18m-8.
\end{equation}
To prove that $S_m$ is zero-sum free, suppose, to the contrary, that it
contains a {\sl nonempty} zero-sum subsequence
$$T=e^c\prod_{j=1}^4q_j^{a_j}\prod_{i=1}^3q_{4+i}^{b_i},   \qquad \text{ where } c\in[0,4m-1] \text{ and } a_j,b_i\in[0,2m-1].$$
Considering the direct summand $\langle z_i\rangle$ of $B_m$, by \eqref{eq:q-last},
we obtain $b_i z_i=0$ from $\sigma(T)=0$, and hence
$m\mid b_i$ for each $i\in [1,3]$.
Considering the direct summand
$\langle y_j\rangle$ of $B_m$, by \eqref{eq:q-first} and \eqref{eq:q-last}, we obtain
\begin{equation}\label{eq:y-coordinate}
 a_j+\sum_{\substack{1\le i\le3\\i\ne j}}b_i
 \equiv0\pmod{2m}, \qquad \text{ where } j\in[1,4].
\end{equation}
Consequently $m\mid a_j$ for every $j$.
Similarly, considering the direct summand $\langle e\rangle$,
we obtain
\begin{equation}\label{eq:e-coordinate}
 c+2\sum_{j=1}^4a_j+2\sum_{i=1}^3b_i
 \equiv0\pmod{6m},
\end{equation}
and therefore $m\mid c$. Write
\begin{equation}\label{eq:c'aj'bi'}
 c=mc',\qquad a_j=ma_j',\qquad b_i=mb_i',
 \qquad \text{ where }
 c'\in [0,3]\quad \text{ and } \quad a_j',b_i'\in\{0,1\}.
\end{equation}
Dividing each congruence in \eqref{eq:y-coordinate} by $m$ gives
\begin{equation}\label{equation:a'j+=0mod2}
a'_j+\sum_{\substack{i\in[1,3]\\i\ne j}}b'_i
\equiv0\pmod2
\qquad\text{for each }j\in[1,4].
\end{equation}

Let $\{v_1,\ldots,v_4\}$ be a basis of $\FF_2^4$.
The four congruences in \eqref{equation:a'j+=0mod2}
are equivalent to the single vector equation
\begin{equation}\label{equation:coordinatEqu}
 \sum_{j=1}^4a'_jv_j
 +b'_1(v_2+v_3+v_4)
 +b'_2(v_1+v_3+v_4)
 +b'_3(v_1+v_2+v_4)=0
 \quad\text{in }\FF_2^4.
\end{equation}
Since every coefficient in \eqref{equation:coordinatEqu} is either zero or one by \eqref{eq:c'aj'bi'}, it follows that the vectors
with coefficient one form a possibly empty zero-sum subset
of the set $\Omega$ in Lemma~\ref{lem:binary}.
Its cardinality is
\begin{equation}\label{eq:h=}
 h=\sum_{j=1}^4a'_j+\sum_{i=1}^3b'_i.
\end{equation}
Lemma \ref{lem:binary} therefore gives $$h\in\{0,4\}.$$
Dividing \eqref{eq:e-coordinate} by $m$ and using
\eqref{eq:h=}, we obtain
$$c'+2h\equiv0\pmod6.$$
If $h=4$, this congruence would give $c'\equiv4\pmod6$,
contrary to $c'\in [0,3]$ given by \eqref{eq:c'aj'bi'}. Hence $h=0$, and the same
congruence then gives $c'=0$.
Since the summands in \eqref{eq:h=} are nonnegative,
$h=0$ forces $a'_j=0$ for every $j$ and $b'_i=0$ for
every $i$. By \eqref{eq:c'aj'bi'}, we have $c=a_j=b_i=0$ and $T$ is the empty sequence, a contradiction.
This proves that $S_m$ is zero-sum free.

Now let $r\ge8$. Choose a decomposition
\begin{equation}\label{eq:decompositionofGmr}
G_{m,r}=B_m\oplus\bigoplus_{i=1}^{r-8}\langle f_i\rangle,
\qquad \ord(f_i)=m\quad\text{for each }i\in[1,r-8].
\end{equation}
Define
\begin{equation}\label{eq:padded-sequence}
 S_{m,r}=S_m\prod_{i=1}^{r-8}f_i^{\,m-1}.
\end{equation}
Since $S_m$ is zero-sum free, it follows from \eqref{eq:decompositionofGmr} and \eqref{eq:padded-sequence} that $S_{m,r}$ is zero-sum free.
By \eqref{eq:lengthofSm}, we have
$$|S_{m,r}|=|S_m|+\sum_{i=1}^{r-8} (m-1) =18m-8+(r-8)(m-1)=(r+10)m-r,$$
which yields $\dd(G_{m,r})\ge(r+10)m-r$.
We check that
$\dd^*(G_{m,r})
 =(r-5)(m-1)+4(2m-1)+(6m-1)
 =(r+9)m-r.$
This gives $\delta(G_{m,r})=\dd(G_{m,r})-\dd^*(G_{m,r})\ge m$. This completes the proof of the theorem.
\end{proof}

By applying Theorem \ref{thm:main} and by $m\geq r$, the inequality $\delta(G_{m,r})\ge m>r-1$ contradicts Conjecture~\ref{conj:upper-bound}. Thus, Conjecture~\ref{conj:upper-bound} is disproved. Moreover, as we show in Section~\ref{sec:rank-exponent}, the Davenport excess $\delta(G)$ can be arbitrarily large even when the rank of $G$ is fixed.

\subsection{A general construction with unbounded Davenport excess}\label{sec:lifting}

Theorem 3.1 gives a direct construction of groups with arbitrarily large Davenport excess. We now present a more general method: starting from a finite abelian group with positive Davenport excess, we construct a family of groups of fixed rank with arbitrarily large Davenport excess. We expect this method to have further applications to related problems in zero-sum theory.

\begin{theorem}\label{thm:lifting}
Let $A=C_{n_1}\oplus\cdots\oplus C_{n_r}$, where
$1<n_1\mid\cdots\mid n_r$.
Let $T=\prod_{i=1}^k g_i^{t_i}$ be a zero-sum free
sequence of length $L$ over $A$, where $t_i\in\NN_0$
and $g_1,\ldots,g_k$ are distinct and generate $A$.
For every integer $m\ge2$, the group
$$A_{m,k}=C_m^{k-r}\oplus
C_{mn_1}\oplus\cdots\oplus C_{mn_r}$$
has a zero-sum free sequence of length $m(L+k)-k$
whose support is a generating set of cardinality $k$.
In particular, $\rr(A_{m,k})=k$ and
$$\delta(A_{m,k})\ge m\bigl(L-\dd^*(A)\bigr).$$
\end{theorem}

\begin{proof}
Let $\epsilon_1,\ldots,\epsilon_k$ be a basis
of $\ZZ^k$, and let
$\varphi:\ZZ^k\rightarrow A$
be the group epimorphism determined by
\begin{equation}\label{equation:var(epsilon)=g}
\varphi(\epsilon_i)=g_i \qquad \text{ for each } i\in [1,k].
\end{equation}
Denote $\Lambda=\ker \varphi$. Then
\begin{equation}\label{equation:A/ZZk=lmabda}
A\cong\ZZ^k/\Lambda
\end{equation} and
the index $[\ZZ^k: \ker \varphi]=|A|$.
By the subgroup theorem for free abelian groups,
$\Lambda$ is free abelian of rank at most $k$.
Since $|A|\mathbb Z^k\subseteq\Lambda$, we have
$$
|A|\epsilon_i\in\Lambda
\qquad\text{for each }i\in[1,k].
$$
The vectors $|A|\epsilon_1,\ldots,|A|\epsilon_k$ are linearly independent over $\mathbb Z$. Hence $\Lambda$ has rank exactly $k$.

Define
$$
B=\mathbb Z^k/m\Lambda.
$$
Since $\varphi$ is surjective, every element of $A$ has the form $\varphi(z)$ for some $z\in\mathbb Z^k$. We define
\begin{equation}\label{equation:iota}
\iota:A\longrightarrow B,\qquad
\iota(\varphi(z))=mz+m\Lambda
\quad\text{for }z\in\mathbb Z^k.
\end{equation}
This map is well defined: if $\varphi(z)=\varphi(z')$, then $z-z'\in\Lambda$, so $m(z-z')\in m\Lambda$.
It is easy to check that $\iota$ is a homomorphism. Moreover, if
$\iota(\varphi(z))=0$, then $mz=m\lambda$ for some
$\lambda\in\Lambda$. Since $\ZZ^k$ is torsion-free,
we obtain $z=\lambda\in\Lambda$, so $\varphi(z)=0$.
Thus $\iota$ is injective.

Since $m\Lambda\subseteq m\mathbb Z^k$, the universal property
of quotient groups
(see \cite[Chapter~I, Theorem~5.6]{Hungerford} for example)
shows that the canonical epimorphism
$\ZZ^k\to\ZZ^k/m\ZZ^k$ induces an epimorphism
$$
\pi:B\longrightarrow\mathbb Z^k/m\mathbb Z^k\cong C_m^k,
\qquad \text{ given by }
\pi(z+m\Lambda)=z+m\mathbb Z^k.
$$
Its kernel is
$$
\ker\pi
=\{mz+m\Lambda:z\in\mathbb Z^k\}
=\iota(A),
$$
where the last equality follows from \eqref{equation:iota} and the surjectivity of $\varphi$. Thus, by the first isomorphism theorem,
$$
B/\iota(A)\cong C_m^k.
$$

Moreover,
$$\pi(\epsilon_i+m\Lambda)=\epsilon_i+m\ZZ^k \qquad \text{ for every } i\in[1,k].$$
These images form a basis of $\ZZ^k/m\ZZ^k$.
Since $\epsilon_1,\ldots,\epsilon_k$ form a basis of $\ZZ^k$
and $m\ge2$, we have $\epsilon_i\notin m\ZZ^k$ for every $i$
and $\epsilon_i-\epsilon_j\notin m\ZZ^k$ whenever $i\ne j$.
As $m\Lambda\subseteq m\ZZ^k$, it follows that the cosets
$\epsilon_1+m\Lambda,\ldots,\epsilon_k+m\Lambda$
are nonzero and pairwise distinct.

Recall that $T=\prod_{i=1}^k g_i^{t_i}$, where $t_i$ is the multiplicity of $g_i$ in $T$. We define a sequence $U_m$ over $B$ by
$$U_m=\prod_{i=1}^k
(\epsilon_i+m\Lambda)^{\,m(t_i+1)-1}.$$

Now we show that $U_m$ is zero-sum free.
Suppose, to the contrary, that $U_m$ has a {\sl nonempty} zero-sum subsequence
$$
V=\prod_{i=1}^k(\epsilon_i+m\Lambda)^{c_i},
$$
where
$$
c_i\in[0,m(t_i+1)-1]\qquad\text{for each }i\in[1,k].
$$
Since $\sigma(\pi(V))=\pi(\sigma(V))=\pi(0)=0$, it follows that
$\sum_{i=1}^k c_i(\epsilon_i+m\ZZ^k)=0,$
and equivalently,
$\sum_{i=1}^k c_i\epsilon_i\in m\ZZ^k.$
Since $\epsilon_1,\ldots,\epsilon_k$ form a basis of $\mathbb Z^k$, we conclude that $m\mid c_i$ for each $i\in[1,k]$. Write
\begin{equation}\label{equation:ci=mai}
c_i=ma_i,\qquad a_i\in\mathbb N_0.
\end{equation}
Since $c_i\le m(t_i+1)-1$, we obtain
$$
a_i\in[0,t_i]\qquad\text{for each }i\in[1,k].
$$
 Since $\sum_{i=1}^k a_i \epsilon_i\in \ZZ^k$,
it follows from \eqref{equation:var(epsilon)=g}, \eqref{equation:iota} and \eqref{equation:ci=mai} that
$$0=\sigma(V)=\sum_{i=1}^k c_i \epsilon_i+m\Lambda
=m\sum_{i=1}^k a_i\epsilon_i+m\Lambda=\iota\left(\varphi(\sum_{i=1}^k a_i \epsilon_i)\right)
=\iota\left(\sum_{i=1}^k a_i g_i\right).$$
Since $\iota$ is injective, it follows that
$\sum_{i=1}^k a_i g_i=0.$
Thus $\prod_{i=1}^k g_i^{a_i}$ is a zero-sum
subsequence of $T$.
As $T$ is zero-sum free, all $a_i$ are zero,
and hence all $c_i$ are zero.
This contradicts the assumption that $V$ is nonempty.
Therefore $U_m$ is zero-sum free.

Since $m\ge2$ and $t_i\ge0$ for each $i\in[1,k]$, the multiplicity $m(t_i+1)-1$ of each element $\epsilon_i+m\Lambda$ in $U_m$ is positive. Hence
$$
\supp(U_m)=\{\epsilon_i+m\Lambda:i\in[1,k]\}.
$$
This set has cardinality $k$ and generates $B$,
since $\epsilon_1,\ldots,\epsilon_k$ generate $\ZZ^k$.
Moreover,
$$|U_m|=\sum_{i=1}^k\bigl(m(t_i+1)-1\bigr)
=m(L+k)-k.$$
Since $g_1,\ldots,g_k$ generate $A$, we have $$k\ge r.$$
Since $\Lambda$ has rank $k$, by the elementary divisors theorem for subgroups of free abelian groups (see \cite[Chapter~II, Theorem~1.6, p.~73]{Hungerford} or \cite[Chapter~III, Theorem~7.8, p.~153]{Lang}), there exists a basis $\{u_1,\ldots,u_k\}$ of $\ZZ^k$ and positive integers $d_1\mid\cdots\mid d_k$ such that
$\{d_1u_1,\ldots,d_ku_k\}$ is a basis of $\Lambda$. Thus,
$\Lambda=d_1\ZZ u_1\oplus\cdots\oplus d_k\ZZ u_k.$
Consequently,
$$\ZZ^k/\Lambda
\cong\bigoplus_{i=1}^k\bigl(\ZZ u_i/d_i\ZZ u_i\bigr)
\cong C_{d_1}\oplus\cdots\oplus C_{d_k}.$$
On the other hand, recall \eqref{equation:A/ZZk=lmabda},
$$\ZZ^k/\Lambda\cong A
=C_{n_1}\oplus\cdots\oplus C_{n_r},
\qquad 1<n_1\mid\cdots\mid n_r.$$
By the uniqueness of the invariant factors, exactly $r$ of the integers $d_1,\ldots,d_k$ are greater than $1$, and these are $n_1,\ldots,n_r$, in this order. Hence
$$
d_i=1\quad\text{for each }i\in[1,k-r],
\qquad
d_{k-r+j}=n_j\quad\text{for each }j\in[1,r].
$$
It follows that
$\{u_1,\ldots,u_{k-r},
n_1u_{k-r+1},\ldots,n_ru_k\}$
forms a basis of $\Lambda$. Consequently,
$\{mu_1,\ldots,mu_{k-r},
mn_1u_{k-r+1},\ldots,mn_ru_k\}$
is a basis of $m\Lambda$, and therefore
$$B=\ZZ^k/m\Lambda
\cong C_m^{k-r}\oplus
C_{mn_1}\oplus\cdots\oplus C_{mn_r}
=A_{m,k}.$$ Observe $\rr(A_{m,k})=k$.
Under this isomorphism, $U_m$ gives a zero-sum free
sequence over $A_{m,k}$ of length $m(L+k)-k$
whose support is a generating set of cardinality $k$.
Finally,
$$\dd^*(A_{m,k})
=(k-r)(m-1)+\sum_{j=1}^r(mn_j-1)
=m\bigl(\dd^*(A)+k\bigr)-k.$$
Since $\dd(A_{m,k})\ge |U_m|$, we obtain
$\delta(A_{m,k})
=\dd(A_{m,k})-\dd^*(A_{m,k})
\ge |U_m|-\dd^*(A_{m,k})
=m\bigl(L-\dd^*(A)\bigr).$ This completes the proof of the theorem.
\end{proof}

We emphasize that the zero-sum free sequence $T$ in Theorem~\ref{thm:lifting} need not have maximum possible length. If $T$ has length $\dd(A)$, then $\Sigma(T)=A\setminus\{0\}$, and hence $\supp(T)$ generates $A$, as shown in the proof of Corollary~\ref{cor:any-seed}. Thus, when seeking the strongest lower bound for the Davenport excess, one may take $T$ to have length $\dd(A)$. However, any zero-sum free sequence of length greater than $\dd^*(A)$ whose support generates $A$ suffices for the construction in Theorem~\ref{thm:lifting}. This observation yields the following corollary.

\begin{corollary}\label{cor:any-seed}
Let $A$ be a finite abelian group with $\delta(A)>0$, and let $T$ be a zero-sum free sequence of length $\dd(A)$ over $A$. Put $k=|\supp(T)|$. Then, for every integer $M>0$, there exists a finite abelian group $G_M$ such that
$$
\rr(G_M)=k
\qquad\text{and}\qquad
\delta(G_M)>M.
$$
\end{corollary}

\begin{proof}
Since $T$ is zero-sum free and has maximum possible length, we have $\Sigma(T)=A\setminus\{0\}$. Note that, if some $a\in A\setminus\{0\}$ were not in $\Sigma(T)$, then adjoining $-a$ to $T$ would produce a longer zero-sum free sequence, contradicting the definition of $\dd(A)$. Consequently, $\supp(T)$ generates $A$.

Choose an integer $m\ge2$ such that $m\delta(A)>M$, and let $G_M=A_{m,k}$ be the group constructed in Theorem~\ref{thm:lifting}. Since $T$ has length $\dd(A)$ and its support generates $A$, that theorem gives
$\rr(G_M)=k$
and
$\delta(G_M)
\ge m\bigl(\dd(A)-\dd^*(A)\bigr)
=m\delta(A)>M.$
This completes the proof.
\end{proof}

We now illustrate the construction using the group $C_2^4\oplus C_6$, for which Baayen and van Emde Boas  \cite{BaayenEmdeBoas1969} established a positive Davenport excess.

\begin{example}\label{rem:seed}
Setting $m=1$ in the construction in the proof of
Theorem~\ref{thm:main} gives a zero-sum free
sequence $T$ of length
$\dd(A)=10$ over $A=C_2^4\oplus C_6$ with $|\supp(T)|=8$.
Observe $\dd^*(A)=9$. Fix an integer $M>1$.
Theorem~\ref{thm:lifting}, with $r=5$ and $k=8$,
therefore gives
$$A_{M,8}=C_M^3\oplus C_{2M}^4\oplus C_{6M},
 \text{ and } \delta(A_{M,8})\geq M.$$ This gives the desired groups of rank eight with unbounded Davenport excess. \end{example}

As noted above, Baayen and van Emde Boas' example $G=C_2^4\oplus C_6$ shows that the equality $\DD(G)=\DD^*(G)$ can fail in rank five. Geroldinger and Schneider \cite{GS92} made an important advance by showing that this failure already occurs in rank four. More precisely, they constructed infinitely many groups of rank four with positive Davenport excess, including $C_3^3\oplus C_6$.

While the construction of $G_{m,r}$ in Theorem~\ref{thm:main} was motivated by Baayen and van Emde Boas' example, the new explicit family of finite abelian groups presented below is motivated by the rank-four examples of Geroldinger and Schneider. This family specializes to $C_3^3\oplus C_6$ when $m=1$. More importantly, the construction shows that the Davenport excess is unbounded among finite abelian groups of every fixed rank at least five.

The following theorem gives an explicit construction showing that the Davenport excess is unbounded among finite abelian groups of every fixed rank at least five.

\begin{theorem}\label{thm:rank-five}
Let $r\ge5$ and $m\ge1$ be integers, and put
$F_{m,r}=C_m^{r-4}\oplus C_{3m}^3\oplus C_{6m}.$
Then  $\delta(F_{m,r})=\dd(F_{m,r})-\dd^*(F_{m,r})\ge m.$
\end{theorem}

\begin{proof}
We first consider
$$F_{m,5}=C_m\oplus C_{3m}^3\oplus C_{6m}.$$
Let $f_1,\ldots,f_5$ generate the respective cyclic direct
summands, and put
\begin{gather*}
g_1=f_1+f_2+2f_3+f_4+5f_5,\qquad
g_2=f_2+f_5,\qquad g_3=f_3+f_5,\\
g_4=f_4,\qquad g_5=f_5,\qquad
g_6=g_2+g_3+g_4+g_5,\qquad g_7=g_4+g_5.
\end{gather*}
We shall show that the sequence
$$S_m=g_1^{3m-1}g_2^{3m-1}g_3^{3m-1}
     g_4^{2m-1}g_5^{2m-1}g_6^{2m}g_7^{2m}$$
is zero-sum free.
Let $V=\prod_{i=1}^7g_i^{c_i}$ be a zero-sum subsequence
of $S_m$. By the definitions of $g_1,\ldots,g_7$, we have
$$
\begin{aligned}
0=\sum_{i=1}^7c_i g_i
&=c_1f_1+(c_1+c_2+c_6)f_2
  +(2c_1+c_3+c_6)f_3\\
&\quad +(c_1+c_4+c_6+c_7)f_4+(5c_1+c_2+c_3+c_5+3c_6+c_7)f_5.
\end{aligned}
$$
Since $f_1,\ldots,f_5$ generate distinct cyclic direct
summands whose orders are all divisible by $m$, it follows that
$$
\begin{aligned}
c_1&\equiv0\pmod m,\\
c_1+c_2+c_6&\equiv0\pmod m,\\
2c_1+c_3+c_6&\equiv0\pmod m,\\
c_1+c_4+c_6+c_7&\equiv0\pmod m,\\
5c_1+c_2+c_3+c_5+3c_6+c_7&\equiv0\pmod m.
\end{aligned}
$$
Substituting $c_1\equiv0\pmod m$ into the next three
congruences yields
$c_2+c_6\equiv c_3+c_6
\equiv c_4+c_6+c_7\equiv0\pmod m.$
Finally, writing the last congruence as
$5c_1+(c_2+c_6)+(c_3+c_6)+(c_5+c_6+c_7)
\equiv0\pmod m$
gives $c_5+c_6+c_7\equiv0\pmod m$. Therefore,
$$
c_1\equiv c_2+c_6\equiv c_3+c_6
\equiv c_4+c_6+c_7\equiv c_5+c_6+c_7
\equiv0\pmod m.
$$
Put
$$
a_i=\left\lfloor\frac{c_i}{m}\right\rfloor
\quad\text{for }i\in[1,5],\qquad
a_6=\left\lceil\frac{c_6}{m}\right\rceil,\qquad
a_7=\left\lceil\frac{c_6+c_7}{m}\right\rceil-a_6.
$$
The identity $\lfloor x\rfloor+\lceil y\rceil=x+y$ when
$x+y$ is an integer gives
\begin{equation}\label{eq:rank-five-multiplicities}
\begin{aligned}
c_1&=ma_1,\\
c_i+c_6&=m(a_i+a_6) &&\text{for }i\in[2,3],\\
c_i+c_6+c_7&=m(a_i+a_6+a_7) &&\text{for }i\in[4,5].
\end{aligned}
\end{equation}
Using $g_6=g_2+g_3+g_4+g_5$ and $g_7=g_4+g_5$,
we first collect the coefficients of $g_1,\ldots,g_5$.
Then \eqref{eq:rank-five-multiplicities} gives
$$\begin{aligned}
0=\sum_{i=1}^{7}c_i g_i
&=c_1g_1+(c_2+c_6)g_2+(c_3+c_6)g_3+(c_4+c_6+c_7)g_4+(c_5+c_6+c_7)g_5\\
&=m\bigl[a_1g_1+(a_2+a_6)g_2+(a_3+a_6)g_3+(a_4+a_6+a_7)g_4+(a_5+a_6+a_7)g_5\bigr]\\
&=m\left[\sum_{i=1}^{5}a_i g_i
+a_6(g_2+g_3+g_4+g_5)+a_7(g_4+g_5)\right]\\
&=m\sum_{i=1}^{7}a_i g_i.
\end{aligned}$$
The multiplicity bounds in $S_m$ give
\begin{equation}\label{equation:aitakesvalues}
a_1,a_2,a_3,a_6,a_7\in\{0,1,2\},\qquad
  a_4,a_5\in\{0,1\}.
\end{equation}
Expanding the definitions of $g_1,\ldots,g_7$ gives
$$
\begin{aligned}
m\sum_{i=1}^7a_i g_i
=m\bigl[&
a_1f_1+(a_1+a_2+a_6)f_2
+(2a_1+a_3+a_6)f_3\\
&+(a_1+a_4+a_6+a_7)f_4+(5a_1+a_2+a_3+a_5+3a_6+a_7)f_5\bigr].
\end{aligned}
$$
Since the orders of $f_2,f_3,f_4$
are $3m$ and the order of $f_5$ is $6m$, it follows that
$$
\begin{aligned}
3m&\mid m(a_1+a_2+a_6),\\
3m&\mid m(2a_1+a_3+a_6),\\
3m&\mid m(a_1+a_4+a_6+a_7),\\
6m&\mid m(5a_1+a_2+a_3+a_5+3a_6+a_7),
\end{aligned}
$$
and so
\begin{equation}\label{eq:rank-five-reduced-equations}
 \begin{aligned}
 a_1+a_2+a_6&=0,& -a_1+a_3+a_6&=0,\\
 a_1+a_4+a_6+a_7&=0,& -a_1+a_5+a_6+a_7&=0
 \end{aligned}
 \quad\text{in }\FF_3.
\end{equation}
Here the fourth equation in \eqref{eq:rank-five-reduced-equations} follows as
$$
\begin{aligned}
0
&=(5a_1+a_2+a_3+a_5+3a_6+a_7)-(a_1+a_2+a_6)-(2a_1+a_3+a_6)\\
&=2a_1+a_5+a_6+a_7=-a_1+a_5+a_6+a_7
\qquad\text{in }\FF_3.
\end{aligned}
$$
The last divisibility relation above also gives
$5a_1+a_2+a_3+a_5+3a_6+a_7\equiv0\pmod6,$
so
$$Q=a_1+a_2+a_3+a_5+a_6+a_7$$
is even. Subtracting the last two equations in
\eqref{eq:rank-five-reduced-equations} gives
$a_4-a_5=a_1$ in $\FF_3$.
If $a_1=1$, by \eqref{equation:aitakesvalues} and the same equations, we have
$a_4=1$, $a_5=0$, $a_2+a_6=2$, $a_3=a_7,$
so $Q=3+2a_7$ is odd, a contradiction.
If $a_1=2$, they give
$a_4=0,$ $a_5=1$, $a_3+a_6=2$, $a_2=a_7,$
so $Q=5+2a_7$ is again odd. Then
\eqref{equation:aitakesvalues} gives
$a_1=0$. Combined with \eqref{eq:rank-five-reduced-equations}, we have $a_2=a_3$, $a_4=a_5$, and
$$a_4+a_6+a_7\in\{0,3\},\qquad
  Q=2a_2+(a_4+a_6+a_7).$$
Since $Q$ is even, $a_4+a_6+a_7=0$. Combined with \eqref{equation:aitakesvalues} and \eqref{eq:rank-five-reduced-equations}, we conclude that $a_i=0$ for each $i\in [1,7]$.
 By \eqref{eq:rank-five-multiplicities} and
nonnegativity, all $c_i=0$. Thus $V$ is empty, proving that
$S_m$ is zero-sum free.

For $r\ge5$, write
$$F_{m,r}=F_{m,5}\oplus\bigoplus_{i=1}^{r-5}\langle h_i\rangle,
\qquad \ord(h_i)=m.$$
The sequence $S=S_m\prod\limits_{i=1}^{r-5}h_i^{m-1}$ is zero-sum
free  and
$|S|=17m-5+(r-5)(m-1)=(r+12)m-r.$
Since $\dd^*(F_{m,r})=(r+11)m-r$, the desired inequality follows.
\end{proof}

Theorem~\ref{thm:rank-five} shows that the supremum in
Liu's Problem~\ref{prob:liu} is infinite, even when restricted
to groups of rank five. Although this answers Liu's question,
it naturally leads to the following problem concerning
the least rank at which unbounded Davenport excess occurs.

\medskip

\begin{problem}\label{prob:least-rank}
Determine the least positive integer $r_0$ for which
$$
\sup_{\operatorname{rank}(G)=r_0}\delta(G)=\infty,
$$ where $G$ ranges over finite abelian groups of rank $r_0$.
\end{problem}

The following proposition restricts the possible values of $r_0$ and shows that unboundedness also holds at every rank strictly larger than $r_0$.

\begin{proposition}\label{prop:least-rank}
The integer $r_0$ in Problem~\ref{prob:least-rank} exists
and satisfies
$$3\le r_0\le5.$$
Moreover, the Davenport excess is unbounded among finite abelian
groups of every fixed rank $r\ge r_0$.
\end{proposition}

\begin{proof}
Theorem~\ref{thm:rank-five} shows that $r_0$ exists and
$r_0\le5$. Since $\DD(G)=\DD^*(G)$ for groups of rank
at most two, we have $r_0\ge3$.

Let $G=C_{n_1}\oplus\cdots\oplus C_{n_r}$, where
$1<n_1\mid\cdots\mid n_r$, and put $H=C_{n_1}\oplus G$.
Note that
$\dd(H)\ge\dd(G)+n_1-1.$
Since $\dd^*(H)=\dd^*(G)+n_1-1$, we obtain
$\delta(H)\ge\delta(G)$.
Since $\rr(H)=r+1$ and $\delta(H)\ge\delta(G)$, the Davenport excess is unbounded among groups of rank $r+1$ whenever it is unbounded among groups of rank $r$. Repeating this argument proves the assertion for every larger rank.
\end{proof}

\section{Disproof of Conjecture~\ref{conj:narkiewicz-sliwa}}\label{sec:narkiewicz}

The Narkiewicz constants arise in quantitative factorization
theory. Let $K$ be an algebraic number field with ring of
integers $\mathcal O_K$ and nontrivial ideal class group $G$.
For $a\in\mathcal O_K\setminus\{0\}$, let $\mathsf Z(a)$ be
the set of factorizations of $a$ into irreducibles, counted
up to order and units. For $k\in\NN$ and $x\ge1$, put
$$\begin{aligned}
F_k(x)=\bigl|\bigl\{a\mathcal O_K:\;&
a\in\mathcal O_K\setminus\{0\},\\
& (\mathcal O_K:a\mathcal O_K)\le x,
\quad |\mathsf Z(a)|\le k\bigr\}\bigr|.
\end{aligned}$$
Thus $F_k(x)$ counts nonzero principal integral ideals of
norm at most $x$ whose generators have at most $k$ factorizations.
There is a constant $c_{K,k}>0$ such that
$$F_k(x)\sim c_{K,k}\,x(\log x)^{-1+1/|G|}
(\log\log x)^{\Nark_k(G)}
\qquad(x\to\infty);$$
see \cite[Theorem~9.3.2]{GHK06} for a more general result
and \cite[Introduction]{GGW11}. In particular, $\Nark_1(G)$
describes the corresponding counting problem for algebraic
integers with unique factorization, counted up to associates.

We now recall the combinatorial definition for an arbitrary
nontrivial finite abelian group $G$. Put
$G^\bullet=G\setminus\{0\}$, and write
$$\mathcal B(G^\bullet)
=\{B\in\cF(G^\bullet):\sigma(B)=0\}$$
for the monoid of zero-sum sequences over $G^\bullet$.
Define the map
$$\tau:\cF(G)\longrightarrow\cF(G\times\NN),\qquad\tau(S)=\prod_{g\in G}\prod_{j=1}^{\vterm_g(S)}(g,j).$$
Thus $\tau(S)$ is a squarefree labeled sequence in which
distinct occurrences of the same group element receive
distinct labels.

For $B\in\mathcal B(G^\bullet)$, let
$\mathsf Z(\tau(B))$ denote the set of factorizations
$$\tau(B)=T_1\cdots T_t,\qquad t\in\NN_0,$$
where each $T_i\in\cF(G^\bullet\times\NN)$ becomes a
minimal zero-sum sequence after its labels are deleted.
Two such factorizations are identified precisely when
they differ only in the order of the factors.
For $B=1$, the empty product is its unique factorization.
For $k\in\NN$, the \emph{$k$th Narkiewicz constant} of $G$
is given by
\begin{equation}\label{eq:narkiewicz-definition}
 \Nark_k(G)
 =\max\bigl\{
     |B|:
     B\in\mathcal B(G^\bullet),\
     |\mathsf Z(\tau(B))|\le k
   \bigr\};
\end{equation}
see \cite[Definition 6.2.1]{GHK06}.

Narkiewicz \cite{N79} gave combinatorial descriptions of
these invariants in 1979. In 1982, Narkiewicz and
\'{S}liwa \cite{NS82} proposed the formula in
Conjecture~\ref{conj:narkiewicz-sliwa}.
Gao, Li, and Peng \cite[Theorem~1.2]{GLP11} proved the
formula for $C_p\oplus C_p$ for every prime $p$, and
Gao, Peng, and Zhong \cite[Theorem~2.3]{GPZ13} proved it
for all rank-two groups. For further developments and
related inverse problems, see
\cite{Gao97,GGW11,GHLLQZ24,FHZ24,FZ25}.

The Davenport constant and the Narkiewicz constants satisfy the following inequalities; see \cite[Proposition~6.2.2]{GHK06}:
\begin{equation}\label{eq:davenport-narkiewicz}
\DD(G)\le\Nark_1(G)\le\cdots\le\Nark_k(G)
\le\Nark_1(G)+(k-1)\DD(G)
\qquad(k\in\NN).
\end{equation}
Since $\sum_{i=1}^r n_i=\dd^*(G)+\rr(G)$, the Narkiewicz--\'{S}liwa conjecture would imply Conjecture \ref{conj:upper-bound}. Thus, our disproof of Conjecture~\ref{conj:upper-bound} also disproves the Narkiewicz--\'{S}liwa conjecture.

More strongly, our construction gives a quantitative measure of this failure. The family $F_{m,r}$ from Theorem~\ref{thm:rank-five} shows that the difference between $\Nark_1(G)$ and its conjectured value $\dd^*(G)+\rr(G)$ can be arbitrarily large even when the rank of $G$ is fixed.

\begin{corollary}[Disproof of the Narkiewicz--\'{S}liwa conjecture]
\label{cor:narkiewicz}
For every integer $r\ge5$, the difference
$$
\Nark_1(G)-\dd^*(G)-\rr(G)
$$
is unbounded as $G$ ranges over finite abelian groups of rank $r$. In particular, the Narkiewicz--\'{S}liwa conjecture fails for groups of every rank $r\ge5$.
\end{corollary}

\begin{proof}
Fix $r\ge5$ and consider the family $F_{m,r}$ from Theorem~\ref{thm:rank-five}, where $m\ge2$. The invariant factors of $F_{m,r}$ sum to $(r+11)m$. By \eqref{eq:davenport-narkiewicz} and Theorem~\ref{thm:rank-five},
$$
\Nark_1(F_{m,r})
\ge \DD(F_{m,r})
\ge (r+12)m-r+1.
$$
Since $\dd^*(F_{m,r})+\rr(F_{m,r})=(r+11)m$, we obtain
$$
\Nark_1(F_{m,r})-\dd^*(F_{m,r})-\rr(F_{m,r})
\ge m-r+1.
$$
The right-hand side tends to infinity as $m\to\infty$ with $r$ fixed. This proves the corollary.
\end{proof}

\section{Disproof of Cross number Conjecture}\label{sec:cross}

Krause \cite{Krause84} introduced the cross number in the study of algebraic number fields. Krause and Zahlten \cite{KZ91} developed the theory further in the context of Krull monoids and their divisor class groups; see also \cite[Chapter~5]{GHK06}. General bounds and asymptotic results were obtained by Girard \cite{Girard09}, and further extremal cases were established by He \cite{He14} and Kim \cite{Kim15}. More recently, Bashir and Schmid \cite{BS25} studied the sets of attainable cross numbers, including their arithmetic progressions and gaps.

In 2008, Girard \cite[Conjecture~1.2]{Girard08} proposed the following conjecture, see also \cite[Conjecture~1]{Girard10}.

\begin{conjecture}[Girard,
{\cite[Conjecture~1.2]{Girard08}}]\label{conj:girard}
Let $G=C_{n_1}\oplus\cdots\oplus C_{n_r}$ be a nontrivial
finite abelian group, where $1<n_1\mid\cdots\mid n_r$.
If $S$ is a zero-sum free sequence over $G$ with
$|S|\ge\dd^*(G)$, then
$$\kk(S)\le\sum_{i=1}^r\frac{n_i-1}{n_i}<r.$$
\end{conjecture}

The conjecture would imply the upper bound $\DD(G)\le\rr(G)\exp(G)$ and the exact formula $\DD(C_n^r)=r(n-1)+1$ \cite[Propositions~2.1 and~2.2]{Girard08}. It is known to hold for cyclic groups, $p$-groups, and groups of rank two \cite[Proposition~2.3 and Theorem~2.4]{Girard08}. On the other hand, examples of Geroldinger, Liebmann, and Philipp \cite[Corollary~3.2 and Remark~3.3]{GLP12} show that the conjectured inequality can fail for zero-sum free sequences of length $\dd^*(G)-1$. In contrast, our construction yields zero-sum free sequences with unbounded cross numbers at length $\dd^*(G)$, even when the rank of the group is fixed.

\begin{theorem}\label{thm:cross}
For every fixed integer $r\ge6$ and every real number $M>0$,
there exist a finite abelian group $G$ of rank $r$ and a
zero-sum free sequence $S$ over $G$ such that
$$|S|=\dd^*(G),\qquad \kk(S)>M.$$
In particular, Conjecture~\ref{conj:girard} is false.
\end{theorem}

\begin{proof}
Fix an integer $d>2M$, and take $d$ distinct primes $p_1,\ldots,p_d$.
Put
$$m=\prod_{i=1}^d p_i,\qquad
  \ell=\sum_{i=1}^d(p_i-1),\qquad
  \kappa=\sum_{i=1}^d\left(1-\frac1{p_i}\right).$$
The inequality $(a-1)+(b-1)\le ab-1$ for positive integers
$a,b$ gives $\ell\le m-1$ by induction. Clearly, $\ell\ge1$.
Let $\langle f\rangle$ be a cyclic group of order $m$, and put
$$R=\prod_{i=1}^d\left(\frac{m}{p_i}f\right)^{p_i-1}$$
Since $\ord((m/p_i)f)=p_i$ and the primes $p_i$ are distinct,
the subgroups $\langle(m/p_i)f\rangle$ are independent.
Consequently, $R$ is zero-sum free and
$$|R|=\ell,\qquad \kk(R)=\kappa\ge d/2>M.$$

By Theorem~\ref{thm:rank-five}, $H=F_{m,r-1}$ has a zero-sum free
sequence $T$ of length $\dd^*(H)+m$.
Since $1\le\ell\le m-1$, choose $V\mid T$ with
$$|V|=\dd^*(H)+m-1-\ell.$$
Now consider $G=H\oplus\langle f\rangle=F_{m,r}$ and put
$S=V\cdot R$. Since $V$ and $R$ are zero-sum free and lie in
independent direct summands, $S$ is zero-sum free. Moreover,
$$|S|=|V|+|R|=\dd^*(H)+m-1=\dd^*(G),\qquad
  \kk(S)\ge\kk(R)=\kappa>M.$$
This completes the proof.
\end{proof}

\begin{remark}\label{rem:cross-threshold}
For every fixed $r\ge6$ and every $L,M\ge0$, there also exist
$G$ and $S$ with
$$|S|-\dd^*(G)>L,\qquad \kk(S)>M.$$
Keep $T$ unchanged from the preceding proof, and choose sufficiently many distinct primes so that $\ell+1>L$ and $\kappa>M$.
Then
$|T\cdot R|-\dd^*(G)=\ell+1>L$
and $\mathsf{k}(T\cdot R)\ge\kappa>M$.
Consequently, no finite-valued functions
$f,h:\NN\to[0,\infty)$ can ensure
$\kk(S)\le h(\rr(G))$ for every zero-sum free $S$ satisfying
$|S|\ge\dd^*(G)+f(\rr(G))$.
\end{remark}

\begin{remark}
The significance of Theorem~\ref{thm:cross} lies in
obtaining unbounded cross numbers at the prescribed
length $\dd^*(G)$ while keeping the rank fixed.
Without a restriction on the sequence length,
unbounded cross numbers already occur over cyclic
groups, as the preceding construction shows.
The sequences constructed here are not asserted to have
maximum possible length $\dd(G)$.
Moreover, Conjecture~\ref{conj:girard} uses invariant factors.
The classical cross-number conjectures
$\kk(G)=\kk^*(G)$ and $\mathsf K(G)=\mathsf K^*(G)$
concern the maximal cross numbers of zero-sum free and
minimal zero-sum sequences, respectively. Their proposed
values use the prime-power elementary divisors and may
themselves grow without bound at fixed rank
\cite[Section~3.8 and Appendix~B, Problem~5]{GGZ26}.
Thus the unboundedness established here does not
by itself contradict these conjectures.
\end{remark}

\section{Bounds in terms of rank and exponent}\label{sec:rank-exponent}

Besides Problem~\ref{prob:rank-bound} introduced in the
Introduction, Liu proposed an earlier problem in the same
direction concerning the Davenport excess.

\begin{problem}[Liu, {\cite{Liu20}}]\label{prob:liu}
Determine $\sup_G\delta(G)/\rr(G)$, where $G$ ranges over all
nontrivial finite abelian groups.
\end{problem}

We answer both Problems~\ref{prob:rank-bound} and
\ref{prob:liu}, and then compare our lower-bound
constructions with the classical upper bound of
van Emde Boas and Kruyswijk.

Using non-dispersive sequences over $C_n^{r-1}\oplus C_{kn}$, Liu \cite[Corollary~4.9]{Liu20} showed that the supremum in Problem~\ref{prob:liu} is at least $1/2$. Our fixed-rank construction shows that the supremum in Problem~\ref{prob:liu} is infinite for every fixed rank $r\ge5$.

\begin{corollary}\label{cor:main}
For every fixed integer $r\ge5$,
$$
\sup_{\rr(G)=r}\delta(G)=\infty.
$$
Thus no function of the rank alone bounds the Davenport excess, and
$\sup_G\frac{\delta(G)}{\rr(G)}=\infty$
even when the supremum is restricted to groups of rank five.
\end{corollary}

\begin{proof}
The first assertion follows by letting $m\to\infty$ in Theorem~\ref{thm:rank-five}. Since the Davenport excess is unbounded among groups of rank five, the second assertion follows as well.
\end{proof}

\begin{proposition}\label{prop:rank-exponent-lower}
For every integer $r\ge5$, there exist infinitely many pairwise nonisomorphic finite abelian groups $G$ of rank $r$ such that
$$
\delta(G)\ge
\frac{\lfloor r/5\rfloor}{6}\exp(G).
$$
\end{proposition}

\begin{proof}
Fix $r\ge5$, and write
$$
r=5q+k,\qquad
q=\left\lfloor\frac r5\right\rfloor,\qquad
k\in[0,4].
$$
For each integer $m\ge2$, put
$$
H_{m,r}=F_{m,5}^{\oplus q}\oplus C_m^k,
$$
where
$F_{m,5}=C_m\oplus C_{3m}^3\oplus C_{6m}$
and $F_{m,5}^{\oplus q}$ denotes the direct sum of $q$ copies of $F_{m,5}$.

By Theorem~\ref{thm:rank-five}, each copy of $F_{m,5}$ admits a zero-sum-free sequence of length at least $\dd^*(F_{m,5})+m$. Taking these sequences together with a zero-sum-free sequence of length $k(m-1)$ over $C_m^k$, we obtain
$\delta(H_{m,r})\ge qm.$
Observe
$\rr(H_{m,r})=r$ and
$\exp(H_{m,r})=6m.$
Consequently,
$$
\delta(H_{m,r})\ge qm=\frac{\lfloor r/5\rfloor}{6}\exp(H_{m,r}).
$$
As $m$ varies, the exponents $\exp(H_{m,r})=6m$ are distinct. Hence the groups $H_{m,r}$ are pairwise nonisomorphic, proving the proposition.
\end{proof}

\begin{remark}
For every $r\ge5$, we have
$\frac{\lfloor r/5\rfloor}{6}\exp(G)\ge\frac1{54}r\exp(G).$
Thus Proposition \ref{prop:rank-exponent-lower} implies that, for every fixed rank $r\ge5$, infinitely many pairwise nonisomorphic finite abelian groups $G$ of rank $r$ satisfy
$$
\delta(G)\ge\frac{\rr(G)\exp(G)}{54}.
$$
The uniform constant $1/54$  can be improved. However, the precise value of this constant is not our main concern here. The purpose of the proposition is to show that the Davenport excess can grow on the scale of $\rr(G)\exp(G)$, even when the rank is fixed.
\end{remark}

Our lower-bound construction shows that the Davenport excess can grow on the scale of $\rr(G)\exp(G)$. It is therefore natural to ask whether the Davenport constant admits a uniform upper bound of the same order, and whether this scale is best possible. To address this question, we first recall the classical upper bound for $\DD(G)$. Put $r=\rr(G)$.
The classical order--exponent estimate,
proved by van Emde Boas and Kruyswijk
\cite[Theorem~7.1]{vEBK69} and rediscovered by Meshulam
\cite[Theorem~1]{Mesh90}, is
$$\DD(G)\le\exp(G)\left(1+\log\frac{|G|}{\exp(G)}\right).$$
It refines the coarse scale $O(\exp(G)\log|G|)$ by using the
quotient $|G|/\exp(G)$. Since $|G|\le\exp(G)^{\rr(G)}$, it gives
$$\DD(G)\le\exp(G)\bigl(1+(\rr(G)-1)\log\exp(G)\bigr);$$
see also
\cite[Theorem~3.4.16]{GGZ26} and
\cite[Theorem~5.5.5]{GHK06}.

Theorem \ref{thm:rank-exponent} in this paper removes the logarithmic factor
from this estimate with the absolute constant $16/5$, simultaneously
for all ranks and exponents. Combining the two bounds gives
$$\DD(G)\le\exp(G)\min\left\{
1+\log\frac{|G|}{\exp(G)},\,\frac{16}{5}r\right\}.$$
The new term is smaller precisely when
$\log(|G|/\exp(G))>16r/5-1$; the classical bound remains sharper,
for example, for cyclic groups. For $G=C_n^r$ with $r\ge2$, the classical upper bound is
$$
\DD(G)\le n\bigl(1+(r-1)\log n\bigr),
$$
whereas our result gives
$$
\DD(G)\le\frac{16}{5}rn.
$$
For each fixed $r\ge2$, the ratio of the classical upper bound to our bound is
$$
\frac{5\bigl(1+(r-1)\log n\bigr)}{16r},
$$
which tends to infinity as the exponent $n\to\infty$.

Other estimates are sharper in particular regimes.
Bhowmik and Schlage-Puchta \cite[Theorem~1.1]{BSP12} proved
$\DD(G)\le\exp(G)+|G|/\exp(G)-1$ when
$\exp(G)\ge\sqrt{|G|}$.
Girard \cite[Theorem~1]{Girard18} proved
$\DD(C_n^r)\sim rn$ as $n\to\infty$ for every fixed $r$;
see also \cite[Theorem~3.7.15]{GGZ26}.
This asymptotic formula is stronger for each fixed rank as
$n\to\infty$, but its fixed-rank formulation does not by itself
provide a single constant valid for all ranks and exponents.
The contribution of Theorem~\ref{thm:rank-exponent} is this
uniformity, together with linear dependence on the rank.

A bound $\delta(G)\le C\rr(G)\exp(G)$ with absolute $C$
is equivalent to an estimate $\DD(C_n^r)\le K rn$
with absolute $K$. The former gives the latter with
$K=C+1$, because $\dd^*(C_n^r)+1\le rn$.
Conversely, a group $G$ of rank $r$ embeds in
$C_{\exp(G)}^r$, so
$$\delta(G)\le\DD(G)\le\DD(C_{\exp(G)}^r)
\le K r\exp(G).$$
Geroldinger, Grynkiewicz, and Zhong \cite[p.~154]{GGZ26}
highlight improving the classical upper estimate for the Davenport
constant as a particularly challenging problem.

\begin{remark} The following theorem provides the first general improvement of the classical upper bound of van Emde Boas and Kruyswijk: it establishes a uniform upper bound for the Davenport constant that is linear in both the rank and the exponent, with an absolute constant independent of the group.
\end{remark}

\begin{theorem}\label{thm:rank-exponent}
For every nontrivial finite abelian group $G$,
$$\DD(G)\le\frac{16}{5}\rr(G)\exp(G).$$
\end{theorem}

To prove the theorem,  we need some
number-theoretic estimates and the elementary
group-algebra facts.

\begin{lemma}\label{ub:lem:lcm}
Let $\psi(x)=\sum\limits_{\ell^a\le x}\log\ell$, where $\ell$ ranges
over primes and $a\ge1$. For integers $n\ge1$,
$\psi(n)=\log\lcm(1,\ldots,n)$. For all real $x>0$,
$$\psi(x)<\min\{1.03883x,\ 1.01624x+1.42620\sqrt{x}\}.$$
\end{lemma}

\begin{proof}
The identity follows by prime factorization of the least common
multiple. Writing $\vartheta(x)=\sum_{\ell\le x}\log\ell$,
Rosser and Schoenfeld \cite[Theorems~9, 12, and~13]{RS62} proved
$\vartheta(x)<1.01624x$, $\psi(x)<1.03883x$, and
$\psi(x)-\vartheta(x)<1.42620\sqrt{x}$ for all $x>0$.
Adding the first and third estimates gives the second bound.
\end{proof}

\begin{remark}
The estimate in Lemma~\ref{ub:lem:lcm} already
suffices to obtain the slightly weaker bound
$\DD(G)\leq 3.21 \rr(G)\exp(G)$
by the argument below without using Lemma \ref{ub:lem:numerical}.  While, Lemma \ref{ub:lem:numerical} allows us to use
the slightly smaller constant $c=16/5$.
Since this improvement would not affect the order of magnitude
$\rr(G)\exp(G)$, we retain the simpler constant $16/5$
for a cleaner statement.
\end{remark}

\begin{lemma}\label{ub:lem:numerical}
Let $c=16/5$, $u=1-1/c$, and $h=c-1-\log c$.
Then, for every real $x\ge2$,
$$
\psi(x)<hx+\log(cx)-u.
$$
\end{lemma}

\begin{proof}
By Lemma~\ref{ub:lem:lcm}, it suffices to check that
$$
\begin{aligned}
f_1(x)&=(1.03883-h)x-\log(cx)+u<0,\\
f_2(x)&=(1.01624-h)x+1.42620\sqrt{x}
          -\log(cx)+u<0.
\end{aligned}
$$
The first function is convex, and direct evaluation gives
$f_1(2)<0$ and $f_1(4000)<0$. Hence $f_1(x)<0$
for $2\le x\le4000$.
For $x\ge4000$, we have $f_2(4000)<0$ and
$f_2'(x)<-0.009$, so $f_2(x)<0$.
This proves the lemma.
\end{proof}

We first recall some basic definitions concerning group algebras.
For a field $K$ and a finite abelian group $A$, write $K[A]$
for its group algebra, with basis $\{X^a:a\in A\}$ and
multiplication $X^aX^{a'}=X^{a+a'}$. The augmentation map is
$$
\varepsilon:K[A]\longrightarrow K,\qquad
\varepsilon\left(\sum_{a\in A}c_aX^a\right)=\sum_{a\in A}c_a.
$$
Its kernel is the augmentation ideal.

We need the following nilpotence bound for augmentation ideals, which is a special case of \cite[Proposition~3.2]{GaoGeroldinger2007}.

\begin{lemma}\label{ub:lem:nilpotence}
Let $H$ be a finite abelian $p$-group and let $K$ be a field
of characteristic $p$. If $I$ is the augmentation ideal of
$K[H]$, then
$I^{\mathsf D(H)}=0.$
\end{lemma}

\begin{proof}
Write $H=\bigoplus_{j=1}^r C_{p^{a_j}}$. Since $\mathsf D(H)=\dd^*(H)+1=1+\sum_{j=1}^r(p^{a_j}-1)$,
by \cite[Proposition 3.2]{GaoGeroldinger2007}, taking $k=0$
and $\ell=\dd^*(H)+1$, we derive that
$I^{\mathsf D(H)}\subseteq pK[H]=0,$
and so $I^{\mathsf D(H)}=0$.
\end{proof}

\begin{lemma}\label{ub:lem:characters}
Let $H$ be a finite abelian $p$-group, let $A$ be a finite
abelian group with $p\nmid|A|$, and let $K=\overline{\FF_p}$.
Writing $\widehat A=\operatorname{Hom}(A,K^\times)$, the map
$$\Phi:K[H\oplus A]\longrightarrow
\prod_{\chi\in\widehat A}K[H],
\qquad
X^{(h,a)}\longmapsto\bigl(\chi(a)X^h\bigr)_\chi$$
is an injective $K$-algebra homomorphism.
\end{lemma}

\begin{proof}
Since $A$ is finite abelian and $K$ is algebraically closed
with $\operatorname{char}(K)\nmid |A|$, standard character
theory gives $|\widehat A|=|A|$ and the orthogonality relation
\begin{equation}\label{equation:takesallchi}
\sum_{\chi\in\widehat A}\chi(a)=
\begin{cases}
|A|,&a=0,\\
0,&a\ne0.
\end{cases}
\end{equation}

For each $\chi\in\widehat A$, define
$$
\Phi_\chi:K[H\oplus A]\longrightarrow K[H],
\qquad
X^{(h,a)}\longmapsto\chi(a)X^h.
$$
Since $\chi(a+a')=\chi(a)\chi(a')$ for all $a,a'\in A$,
each $\Phi_\chi$ is a $K$-algebra homomorphism. Therefore,
the product map
$$
\Phi=(\Phi_\chi)_{\chi\in\widehat A}:
K[H\oplus A]\longrightarrow
\prod_{\chi\in\widehat A}K[H]
$$
is a $K$-algebra homomorphism.

To prove injectivity, let
$$F=\sum_{h\in H,\ a\in A}c_{h,a}X^{(h,a)}
\in\ker\Phi.$$
For every $\chi\in\widehat A$, we have
$\Phi_\chi(F)=0$, so
$$
\sum_{h\in H}
\left(\sum_{a\in A}c_{h,a}\chi(a)\right)X^h=0.
$$
Since $\{X^h:h\in H\}$ is a $K$-basis of $K[H]$, we obtain
\begin{equation}\label{equation:takesalla=0}
\sum_{a\in A}c_{h,a}\chi(a)=0
\qquad\text{for every }h\in H\text{ and }\chi\in\widehat A.
\end{equation}

Fix $h\in H$ and $a_0\in A$. Multiplying by $\chi(-a_0)$
and summing over all characters, by \eqref{equation:takesallchi} and \eqref{equation:takesalla=0}, we derive that
$$0=\sum_{\chi\in\widehat A}\sum_{a\in A}
c_{h,a}\chi(a-a_0)
=\sum_{a\in A}c_{h,a}
\sum_{\chi\in\widehat A}\chi(a-a_0)=|A|c_{h,a_0}.$$
The scalar $|A|$ is nonzero in $K$, so $c_{h,a_0}=0$.
Since $h$ and $a_0$ were arbitrary, every coefficient of
$F$ is zero. Hence $\ker\Phi=\{0\}$, proving injectivity.
\end{proof}

Now we are in a position to prove  Theorem~\ref{thm:rank-exponent}.

\begin{proof}[Proof of Theorem~\ref{thm:rank-exponent}]
Write $r=\rr(G)$, and retain
$c,u,h$ from Lemma~\ref{ub:lem:numerical}. For a prime $\ell$, let
$v_\ell(\exp(G))$ denote the exponent of $\ell$ in the prime factorization
of $\exp(G)$. Choose a prime $p\mid \exp(G)$ such that
$$P=p^{v_p(\exp(G))}=\max_{\ell\mid \exp(G)}\ell^{v_\ell(\exp(G))},$$
where the maximum is over prime divisors $\ell$ of $\exp(G)$.
Let $H$ be the Sylow $p$-subgroup of $G$ and $A$ its complementary
Hall subgroup, i.e.,
$$G=H\oplus A,\qquad
\exp(H)=P,\qquad
\exp(A)=Q=\exp(G)/P,\qquad
p\nmid|A|.$$
We may assume that $$Q\geq 2$$ since otherwise, $G$ is a $p$-group and $\DD(G)=\DD^*(G)\leq \rr(G)\exp(G)$ and $\delta(G)=0$ is well-known.
Every exact prime-power factor $\ell^{v_\ell(\exp(G))}$ of $\exp(G)$ is at most
$P$, and these factors are pairwise coprime. Therefore
$$\exp(G)\mid\lcm(1,\ldots,P).$$

By Lemma~\ref{ub:lem:lcm},
$\log Q=\log \exp(G)-\log P\le\log\operatorname{lcm}(1,\ldots,P)-\log P=\psi(P)-\log P$.
Since $A\cong G\diagup H$ is a quotient group of $G$, it can be generated by at most $r$
elements. Its exponent is $Q$, so $A$ is a quotient of $C_Q^r$.
Consequently, $|A|\le Q^r$, and hence
\begin{equation}\label{ub:eq:hall-size}
\log|A|\le r\log Q\le r(\psi(P)-\log P).
\end{equation}

Write
$$H=\bigoplus_{j=1}^r C_{p^{a_j}},
\qquad 0\le a_j\le v_p(\exp(G)),$$
where trivial cyclic factors are included if necessary.
Since $\exp(H)=P$, we have
$$\max_{1\le j\le r}p^{a_j}=P.$$
Consequently,
\begin{equation}\label{ub:eq:b-bound}
\dd(H)=\DD(H)-1=\sum_{j=1}^r(p^{a_j}-1)
\le r(P-1).
\end{equation}
Let $K=\overline{\FF_p}$, and let $I$ be the augmentation
ideal of $K[H]$. Lemma~\ref{ub:lem:nilpotence} gives
\begin{equation}\label{ub:eq:nilpotence}
I^{\DD(H)}=0.
\end{equation}

To prove the conclusion, we suppose, to the contrary, that there exists a zero-sum free sequence
$T=g_1\cdots g_L\in \mathcal{F}(G)$ of length $$L=\dd(G)\ge \lfloor cr\exp(G)\rfloor.$$ Write $g_i=(h_i,a_i)\in H\oplus A$ for each $i\in [1,L]$.
Let
$$\Omega=\{\xi\in K^\times:\xi^Q=1\}.$$
Because $p\nmid Q$, the set $\Omega$ has exactly $Q$
elements. Choose independent random variables $\xi_1,\ldots,\xi_L$
taking values in $\Omega$, each with the uniform distribution,
that is,
$$
\Pr(\xi_i=\omega)=\frac1Q
\qquad\text{for each }i\in[1,L]\text{ and }\omega\in\Omega.
$$
For a character
$\chi\in\widehat A$, define
$$Z_\chi=\#\{i\in[1,L]:\xi_i=\chi(a_i)\}.$$

Now fix $\chi\in\widehat A$.
Since $\exp(A)=Q$, we have
$
\chi(a_i)^Q=\chi(Qa_i)=\chi(0)=1$,
and so $$\chi(a_i)\in\Omega
\quad\text{for all }i\in[1,L].
$$
Since $\xi_i$ is uniformly distributed on $\Omega$ and
$|\Omega|=Q$, it follows that
$$
\Pr\bigl(\xi_i=\chi(a_i)\bigr)=\frac1Q
\qquad\text{for all }i\in[1,L].
$$
Each event $\{\xi_i=\chi(a_i)\}$ depends only on the
corresponding random variable $\xi_i$.
Since $\xi_1,\ldots,\xi_L$ are independent, these
matching events are mutually independent.

We now estimate the expectation
$\mathbb E[c^{-Z_\chi}]$ of the random variable
$c^{-Z_\chi}$:
\begin{equation}\label{equation:E}
\mathbb E=\sum_{k=0}^{L}c^{-k}\Pr(Z_\chi=k).
\end{equation}
Since $c>1$, we have $c^{-k}\ge c^{-\dd(H)}$ whenever $k\le \dd(H)$.
Since $L=\dd(G)\ge \dd(H)$ and all terms in this sum are nonnegative, it follows that
\begin{equation}\label{equation:forE}
\mathbb E
\ge\sum_{k\in [0,\dd(H)]}
c^{-k}\Pr(Z_\chi=k)\ge
c^{-\dd(H)}\sum_{k\in [0,\dd(H)]}
\Pr(Z_\chi=k)=c^{-\dd(H)}\Pr(Z_\chi\le \dd(H)).
\end{equation}
Since the matching events $\{\xi_i=\chi(a_i)\}$ are
independent and each has probability $1/Q$, the random
variable $Z_\chi$ has a binomial distribution with
parameters $L$ and $1/Q$. Thus, for every integer $k\in [0, L]$,
$$
\Pr(Z_\chi=k)
=\binom{L}{k}\left(\frac1Q\right)^k
\left(1-\frac1Q\right)^{L-k}.
$$
Combined with \eqref{equation:E}, we have
$$
\begin{aligned}
\mathbb E
&=\sum_{k=0}^{L}c^{-k}\Pr(Z_\chi=k)\\
&=\sum_{k=0}^{L}c^{-k}\binom{L}{k}
  \left(\frac1Q\right)^k
  \left(1-\frac1Q\right)^{L-k}\\
&=\sum_{k=0}^{L}\binom{L}{k}
  \left(\frac1{cQ}\right)^k
  \left(1-\frac1Q\right)^{L-k}\\
&=\left(1-\frac1Q+\frac1{cQ}\right)^L\\
&=\left(1-\frac{u}{Q}\right)^L,
\end{aligned}
$$
where the last equality uses $u=1-1/c$.
Multiplying the lower bound for $\mathbb E$ given in \eqref{equation:forE}
by $c^{\dd(H)}$ and using $1-x\le\exp(-x)$, we obtain
\begin{equation}\label{equation:tail}
\Pr(Z_\chi\le \dd(H))
\le c^{\dd(H)}\mathbb E
=c^{\dd(H)}\left(1-\frac{u}{Q}\right)^L
\le\exp\left(\dd(H)\log c-\frac{uL}{Q}\right).
\end{equation}

The estimate \eqref{equation:tail} holds for every
$\chi\in\widehat A$. Since $|\widehat A|=|A|$,
$\dd(H)\le r(P-1)$, and
$$
L\ge \lfloor cr\exp(G)\rfloor
=\lfloor crPQ\rfloor\ge crPQ-1,
$$
using $uc=c-1$, it follows from
\eqref{ub:eq:hall-size},
\eqref{ub:eq:b-bound} and \eqref{equation:tail}
that
\begin{align}
\Pr\bigl(\exists\chi\in\widehat A:Z_\chi\le\dd(H)\bigr)
&\le\sum_{\chi\in\widehat A}
\Pr\bigl(Z_\chi\le\dd(H)\bigr)\notag\\
&\le\sum_{\chi\in\widehat A}
\exp\left(\dd(H)\log c-\frac{uL}{Q}\right)
\notag\\
&=|A|\exp\left(\dd(H)\log c-\frac{uL}{Q}\right)
\notag\\
&\le\exp\left(
r\bigl(\psi(P)-\log P+(P-1)\log c-(c-1)P\bigr)
+\frac{u}{Q}\right)\notag\\
&=\exp\left(
r\bigl(\psi(P)-hP-\log(cP)\bigr)
+\frac{u}{Q}\right)<1.
\label{ub:eq:unionnew}
\end{align}
The last equality uses $h=c-1-\log c$, and the term $u/Q$ comes from the inequality $L\ge crPQ-1$.
Since $P\ge2$, Lemma~\ref{ub:lem:numerical} gives
$$
\psi(P)-hP-\log(cP)<-u.
$$
Consequently,
$$
r\bigl(\psi(P)-hP-\log(cP)\bigr)+\frac{u}{Q}
<-ru+\frac{u}{Q}<0,
$$
where the last inequality follows from
$r\ge1$, $Q\ge2$, and $u>0$.

Since $Z_\chi$ is integer-valued and
\eqref{ub:eq:unionnew} gives
$
\Pr\bigl(\exists\chi\in\widehat A:
Z_\chi\le\dd(H)\bigr)<1,$
it follows from $\DD(H)=\dd(H)+1$ that
$$
\Pr\bigl(Z_\chi\ge\DD(H)
\text{ for every }\chi\in\widehat A\bigr)
=1-\Pr\bigl(\exists\chi\in\widehat A:
Z_\chi\le\dd(H)\bigr)>0.
$$

Consequently, there exists a choice of
$\xi_1,\ldots,\xi_L$ such that
\begin{equation}\label{equation:ZchigeqDH}
Z_\chi\ge\DD(H)
\qquad\text{for every }\chi\in\widehat A.
\end{equation}
Fix such a choice for the remainder of the proof.
Consider
$$F=\prod_{i=1}^L(X^{g_i}-\xi_i)\in K[G].$$
Under the homomorphism $\Phi$ from
Lemma~\ref{ub:lem:characters}, the $\chi$-component of
$\Phi(F)$ is
$$\prod_{i=1}^L\bigl(\chi(a_i)X^{h_i}-\xi_i\bigr).$$
Every matching index satisfies $\xi_i=\chi(a_i)$ and
therefore contributes a factor
$$
\chi(a_i)X^{h_i}-\xi_i
=\chi(a_i)(X^{h_i}-1)\in I.$$
By
\eqref{equation:ZchigeqDH}, we know that there are at least $\DD(H)$ such factors. Since $K[H]$
is commutative, their product belongs to
$I^{\DD(H)}=0$ by \eqref{ub:eq:nilpotence}. Hence
$$
\prod_{i=1}^L\bigl(\chi(a_i)X^{h_i}-\xi_i\bigr)=0
\qquad\text{for every }\chi\in\widehat A.
$$
Thus $\Phi(F)=0$. The injectivity of $\Phi$ gives $F=0$.

On the other hand, expanding the product gives
$$F=\sum_{J\subseteq\{1,\ldots,L\}}
(-1)^{L-|J|}
\left(\prod_{i\notin J}\xi_i\right)
X^{\sum_{j\in J}g_j}.$$
Since $T$ is zero-sum free, only $J=\varnothing$
contributes to the coefficient of $X^0$ in $F$.
This coefficient is
$$(-1)^L\prod_{i=1}^L\xi_i,$$
which is nonzero because every $\xi_i$ is nonzero.
Thus $F\ne0$, contradicting $F=0$.
Therefore every sequence of length $\lfloor cr\exp(G)\rfloor$
has a nonempty zero-sum subsequence, and hence
$\DD(G)\le\lfloor cr\exp(G)\rfloor
\le\frac{16}{5}r\exp(G),$ completing the proof of the theorem.
\end{proof}

\section*{Acknowledgements}
The author is deeply grateful to Professor Alfred Geroldinger
for his many insightful comments and valuable suggestions
on both earlier versions and the present version of this paper,
which have substantially improved its content and presentation,
and for drawing his attention to important references and problems. This work was supported by the National Natural Science Foundation
of China (grant nos. 12371335 and 12271520).

\end{document}